\documentclass[11pt,a4paper]{amsart}
\usepackage[left=25mm,top=30mm,bottom=25mm,right=25mm]{geometry}
\usepackage{mathtools,dsfont,amssymb}
\usepackage{makecell}
\usepackage{booktabs}
\usepackage[utf8]{inputenc}                     
\usepackage[english]{babel}                     
\usepackage[fixlanguage]{babelbib}
\usepackage{cmap}                               
\usepackage{graphicx}
\usepackage[all,rotate]{xy}
\usepackage{tikz}
\usepackage[pdfdisplaydoctitle = true,
      colorlinks = true,
      urlcolor = blue,
      citecolor = blue,
      linkcolor = blue,
      pdfstartview = FitH,
      pdfpagemode = UseNone,
      bookmarksnumbered = true,
      unicode=true]{hyperref}                   
\usepackage{todonotes}

\graphicspath{{figs/}}
\newtheorem{thm}{Theorem}[section]
\newtheorem{cor}[thm]{Corollary}
\newtheorem{lem}[thm]{Lemma}
\newtheorem{prop}[thm]{Proposition}
\theoremstyle{definition}
\newtheorem{defn}[thm]{Definition}
\theoremstyle{remark}
\newtheorem{rem}[thm]{Remark}
\newtheorem{exmp}[thm]{Example}

\newcommand{\CC}{\mathbb{C}}                
\newcommand{\RR}{\mathbb{R}}                
\newcommand{\ZZ}{\mathbb{Z}}                

\newcommand{\HH}{\mathbb{H}}                

\newcommand{\GL}{\mathrm{GL}}               
\newcommand{\OO}{\mathrm{O}}                
\newcommand{\UU}{\mathrm{U}}                
\newcommand{\SO}{\mathrm{SO}}               
\newcommand{\SL}{\mathrm{SL}}               

\DeclareMathOperator{\Ad}{Ad}
\DeclareMathOperator{\tr}{tr}

\newcommand{\set}[1]{\left\{#1\right\}}     

\hypersetup{
	pdftitle={On the isotropy stratification of a real representation of a compact Lie group}, 
    pdfauthor={P. Azzi, R. Desmorat, J. Grivaux \& B. Kolev}, 
	pdfsubject={MSC 2020: 20G20, 57S15, 20G05, 14R20, 14L30, 15A72},
    pdfkeywords={Representation of a compact Lie group, Complexification of a compact Lie group, Isotropy stratification, Algebraicity of isotropy strata, Rationality problem},
    pdflang=en
}

\begin{document}

\title[Isotropy strata]{On the isotropy stratification of a real representation of a compact Lie group}

\author{P. Azzi}
\address[Perla Azzi]{Université Paris-Saclay, ENS Paris-Saclay, CentraleSupélec, CNRS, LMPS - Laboratoire de Mécanique Paris-Saclay, 91190, Gif-sur-Yvette, France}
\address{Sorbonne Université and Université Paris Cité, CNRS, IMJ-PRG, F-75005 Paris, France.}
\email{perla.azzi@ens-paris-saclay.fr}

\author{R. Desmorat}
\address[Rodrigue Desmorat]{Université Paris-Saclay, ENS Paris-Saclay, CentraleSupélec, CNRS, LMPS - Laboratoire de Mécanique Paris-Saclay, 91190, Gif-sur-Yvette, France}
\email{rodrigue.desmorat@ens-paris-saclay.fr}

\author{J. Grivaux}
\address[Julien Grivaux]{Sorbonne Université, Université Paris Cité, CNRS, IMJ-PRG, F-75005 Paris, France
}
\email{julien.grivaux@imj-prg.fr}

\author{B. Kolev}
\address[Boris Kolev]{Université Paris-Saclay, ENS Paris-Saclay, CentraleSupélec, CNRS, LMPS - Laboratoire de Mécanique Paris-Saclay, 91190, Gif-sur-Yvette, France}
\email{boris.kolev@ens-paris-saclay.fr}

\date{\today}%

\thanks{The authors were partially supported by CNRS Projet 80--Prime GAMM (Géométrie algébrique complexe/réelle et mécanique des matériaux).}


\begin{abstract}
The aim of the present paper is to provide a comprehensive introduction to some algebraic and geometric aspects of real representations of compact Lie groups, as well as some results concerning isotropy strata and restriction of invariants.
\end{abstract}

\maketitle


\section{Introduction}

Let $\rho \colon K \rightarrow \mathrm{GL}(W)$ be a continuous and finite-dimensional linear representation of  a compact Lie group $K$ on a real vector space $W$. Despite the appearances, the object $(\rho, K)$ is essentially an object of algebraic geometry. The first result leading to this observation, which has been rediscovered several times by different authors, is the following:
\begin{thm}[{\cite[Appendix C]{AS1983}, \cite[{\S 5.1, prop. 1}]{Ser1993}}] \label{thm:algorb}
  Given a real representation $(W, \rho)$ of a compact Lie group, all orbits of $\rho$ are real algebraic subsets of $W$. Equivalently, polynomial invariants of $(W, \rho)$ separate the orbits of $\rho$.
\end{thm}
By Peter-Weyl's theorem, $K$ admits a faithful finite-dimensional linear representation. Hence, $K$ can be considered as a matrix group, that is a closed subgroup of the general linear group $\GL(n; \RR)$ for some integer $n$. It follows from theorem \ref{thm:algorb} that $K$ itself is an algebraic set in $\mathrm{M}(n; \RR)$. Another consequence of Peter-Weyl's theorem is that not only $K$ itself is algebraic, but all finite dimensional representations of $K$ are polynomial. This means that their coefficients are restrictions on $K$ of polynomial functions on $\mathrm{M}(n; \RR)$.
This follows directly from the fact that the algebra of representative functions of $K$ is generated by the matrix coefficients of any faithful representation \cite[prop. III.4.2]{BD1985}.
\par
Algebraicity of orbits does not imply that the orbit of \textit{any} algebraic subset of $W$ is algebraic. For instance, the orbit of the the circle $\mathcal{C}(1,1)$ in $\mathbb{R}^2$ under the natural action of $\mathrm{SO}(2; \mathbb{R})$ is the closed disc $\overline{\mathrm{D}}(0,2)$ which is not algebraic. If we stay in the realm of real algebraic geometry, real algebraic sets are usually hard to produce, as opposed to semialgebraic sets (which are sets in $\mathbb{R}^N$ given by a finite number of polynomial equalities and inequalities). The reason relies on the Tarski-Seidenberg theorem, which states that semialgebraic sets are stable by direct image under polynomial mappings. However, it turns out that it is possible to produce other algebraic sets attached to a finite-dimensional compact representation, which are the isotropy strata that we introduce now.
\par
The set of conjugacy classes of closed subgroups of a compact group admits a partial ordering given by the containment relation : $[L_{1}]\preceq [L_{2}]$ if and only if $L_{1}$ is conjugate to a subgroup of $L_{2}$.
If $\rho \colon K \rightarrow \mathrm{GL}(W)$ is a continuous and finite-dimensional real linear representation of a compact Lie group $K$,
the set of conjugacy classes of stabilizers $[K_w]$ of vectors $w$ of $W$ is a finite poset (cf. remark \ref{rem:slice}), whose elements are called orbit types or symmetry types of $\rho$ . It has a minimal element (cf. theorem \ref{thm:principal}), called the generic isotropy type, and a maximal element $[K]$, corresponding to the symmetry of the null vector. To each orbit type $[L]$ is attached a stratum $\Sigma_{[L]}$, which is the set of vectors whose isotropy group has class $[L]$. The $\Sigma_{[L]}$ define a stratification of $W$, and it is possible to prove that the closure of $\Sigma_{[L]}$ is the $K$-orbit of the vector subspace $W^L$ of vectors fixed by a representative $L$ of $[L]$.

\begin{exmp}\label{ex:facile}
  Consider the natural representation of $\SO(3;\RR)$ on the vector space $\HH^{2}(\RR^{3})$ of traceless symmetric $3 \times 3$ matrices with real coefficients given by $\rho(g)A=gAg^{-1}$. For this representation, there are exactly three orbit types which are totally ordered. The invariant algebra $\RR[\mathbb{H}^{2}(\RR^{3})]^{\SO(3;\RR)}$ is the free algebra generated by the polynomial invariants $I_{2} := \tr A^2$ and $I_3 := \tr A^3$.
  \begin{table}[ht]
    \begin{tabular}{ccc}
      \toprule
      \textbf{Description}                & \textbf{Orbit type}                                    & \textbf{Closure of isotropy stratum} \\
      \midrule
      \textrm{three distinct eigenvalues} & $\mathbb{Z}/2\mathbb{Z} \times \mathbb{Z}/2\mathbb{Z}$ & $\HH^{2}(\RR^{3})$                   \\
      \textrm{two distinct eigenvalues}   & $\OO(2;\mathbb{R})$                                    & $I_2^3 - 6I_3^2=0$                   \\
      \textrm{zero matrix}                & $\SO(3;\mathbb{R})$                                    & \{0\}                                \\
      \bottomrule
    \end{tabular}
  \end{table}
\end{exmp}

The study of isotropy stratification of representations of $\mathrm{SO}(3; \RR)$ on spaces of tensors in classical solid mechanics has been widely investigated, in particular for the elasticity tensor and some of its components (see \cite{FV1996}, \cite{AKP2013}, \cite{OKDD2021}), and recently for the piezoeletricity tensor \cite{Hubert_Jalard_2025}. The results of  \cite{AKP2013} and their later extension in \cite{OKDD2021} by different methods proves that the closures of all isotropy strata considered in their examples are closed algebraic sets. This phenomenon also occurs in our toy example \ref{ex:facile} : the stratum $\overline{\Sigma}_{[\mathrm{O}(2;\mathbb{R})]}$ is a closed algebraic subset of $\mathbb{H}^2(\mathbb{R}^3)$. The authors conjecture in the introduction that this is always the case for a general linear representation of a compact Lie group. In this paper, we explain why this conjecture is true.

\begin{thm}\label{thm:first-main-theorem}
  Let $\rho \colon K \rightarrow \mathrm{GL}(W)$ be a continuous and finite-dimensional real linear representation of a compact Lie group $K$, and let $[L]$ be an orbit type of $\rho$. Then the closure of $\Sigma_{[L]}$ is a real algebraic subset of $W$, which can be defined by polynomial equations in the invariants of $\rho$.
\end{thm}

Describing the isotropy strata is usually achieved using normal forms. Given an orbit type represented by a symmetry group $L\subset K$, the set $W^L \cap \Sigma_{[L]}$ may be considered as normal form for this orbit type. However, there is still an ambiguity left since each $K$-orbit in $\Sigma_{[L]}$ can intersect $W^L$ in more than one point. This ambiguity is encoded in the action of the monodromy group $\Gamma_L = N_K(L)/L$ on $W^L$, where $N_K(L)$ is the normalizer of $L$ in $K$. The group $\Gamma_{L}$ acts freely and transitively on the set of possible normal forms of any element in the isotropy stratum. In particular the natural map
\[
  (\Sigma_{[L]} \cap W^L)/\Gamma_{L} \rightarrow \Sigma_{[L]} / K
\]
is bijective. Since $\Sigma_{[L]}$ is not affine in general, it is more pertinent to consider the map
\[
  W^L / \Gamma_{L} \hookrightarrow \overline{\Sigma}_{[L]} / K
\]
which corresponds at the level of polynomial functions to the map
\begin{equation} \label{restriction}
  \mathrm{Im}\,\left(\mathbb{R}[W]^K \rightarrow \mathbb{R}[W^L]^{\Gamma_L}\right) \rightarrow \mathbb{R}[W^L]^{\Gamma_L}
\end{equation}
between two finitely generated $\mathbb{R}$-algebras. The investigation of the map \eqref{restriction} is extremely interesting. In some cases, this map is an isomorphism.

\begin{exmp}[Chevalley restriction theorem] \label{ex:chev}
  If $K$ is a connected compact Lie group with Lie algebra $\mathfrak{k}$, all maximal tori are conjugate (see \cite[ch. IV, thm 1.6]{BD1985}) and they are maximal abelian subgroups. Let $r$ denote the rank of $K$, which is the dimension of any maximal torus $T$ in $K$. Recall that an element $x$ of $\mathfrak{k}$ is \textit{regular} if the closure of the $1$-parameter subgroup of $K$ generated by $x$ is a maximal torus in $K$. Regular points form a dense open subset $\mathfrak{k}^{\mathrm{reg}} $ of $\mathfrak{k}$, and for any $x$ in $\mathfrak{k}^{\mathrm{reg}}$, $K_x$ is the centralizer of the maximal torus of $K$ containing the $1$-parameter subgroup generated by $x$, which is the torus itself. Hence, if $T$ is a maximal torus with Lie agebra $\mathfrak{t}$, if $W=N_K(T)/T$ is the corresponding Weyl group, the restriction morphism $\RR[\mathfrak{k}]^K \rightarrow \RR[\mathfrak{t}]^W$ is an isomorphism (see \cite[ch. VI, prop. 2.1]{BD1985}).
\end{exmp}

In \cite{AKP2013} the authors provided parametric and then implicit equations of isotropy strata using this map. In their examples, the restriction morphism \eqref{restriction} becomes an isomorphism after passing to the field of fractions. This strategy makes it possible to provide algebraic equations for the closed stratum $\overline{\Sigma}_{[L]}$, but the authors could not always avoid inequations in this description. The reason is that describing $\overline{\Sigma}_{[L]}$ in this way amounts to describe $\overline{\Sigma}_{[L]} / K$, which is birational to $W^{L}/ \Gamma_{L}$. The quotient $W^{L}/ \Gamma_{L}$ is completely explicitable in general thanks to Procesi-Schwarz celebrated result \cite{PS1985}, but it has the major drawback to be semialgebraic as soon as $\Gamma_{L}$ is nontrivial. Going back to the restriction morphism \eqref{restriction}, we cannot expect this map to be an isomorphism in full generality, because the right-hand side is a normal algebra (i.e. integrally closed in its field of fractions), which is not always the case for the left-hand side.

\begin{thm}\label{thm:second-main-theorem}
  If $L$ is an isotropy subgroup of a continuous and linear finite-dimensional real representation of a Lie group $K$, then the restriction map \eqref{restriction} is a normalization morphism. In particular, it induces an isomorphism when restricted to the fields of fractions.
\end{thm}

\begin{exmp}
  We go back to our toy example~\ref{ex:facile}. A normal form for the orbit type $[\OO(2; \mathbb{R})]$ is provided by the linear subspace
  \[
    \HH^{2}(\RR^{3})^{\OO(2; \RR)} = \set{A\in \HH^{2}(\RR^{3}),\;\forall g \in \OO(2; \RR), \;gAg^{-1}=A}= \set{\mathrm{diag}(-\lambda,-\lambda,2\lambda);\; \lambda\in \RR}.
  \]
  The monodromy group is trivial, so $\eqref{restriction}$ is given by the map
  \begin{align*}
    \mathbb{R}[x, y]/(x^3-6y^2) & \rightarrow \mathbb{R}[\lambda]   \\
    P(x, y)                     & \mapsto P(6\lambda^2, 6\lambda^3)
  \end{align*}
  which is the normalization of a cusp. At the level of fractions, $\lambda = y/x$.
\end{exmp}

Related results can be found in the literature, over an algebraically closed field (\cite{Lun1972}, \cite{LR1979}, \cite{Popov1992}). Many of them are due to Luna or rely on his slice theorem. As pointed out by an anonymous referee of a first version of this paper, theorems \ref{thm:first-main-theorem} and \ref{thm:second-main-theorem} have a counterpart for complex linear representations. Our strategy consists in studying the complexified representation of $\rho \colon G \rightarrow \mathrm{Gl}(V)$, where $V=W^{\CC}$ and $G=K^{\CC}$, and to be able to go back from the complex to the real context. A first part of this work is developed in \cite[prop. 5.8]{Sch1980} and implies easily theorem \ref{thm:first-main-theorem}. However, the proof provided in \textit{loc.cit.} is flawed at some point (see remark \ref{rem:sink}), even though the result itself is correct. For our second result, its complexified version is exactly \cite[\S 2, prop.]{Lun1975}, but to get theorem \ref{thm:second-main-theorem}, it is necessary to prove that the normalizer of a compact subgroup of a compact Lie group behaves well under complexification. This technical step is still nontrivial, and is corroborated by \cite[lem. 1.1]{LR1979}. Although these types of methods are certainly well known to a few specialists, we regret that such fundamental and simple results as theorems \ref{thm:first-main-theorem} and \ref{thm:second-main-theorem} appear nowhere in this form in the literature. Unfortunately, this state of affairs has misled many researchers (including ourselves) who reproved many partial results because the lack of accessibility of the existing literature. Our aim is to fill this gap.

The main difficulty arising when writing this article is that it involves techniques from a variety of different fields, mainly on one side differential geometry and representation theory of Lie groups, and on the other side algebraic geometry, algebraic groups and geometric invariant theory. Some results and constructions are known in both contexts. One of the most famous is the existence of slices for smooth proper actions, which implies finiteness of orbit types for smooth actions of compact Lie groups. The algebraic counterpart is Luna's slice theorem for actions of reductive groups on complex affine varieties, which is technically more involved. Another one is complexification, which is straightforward for real algebraic groups, but which is more delicate on the differential-geometric side for compact Lie groups. Although there is an abundant literature on action of Lie groups on manifolds (see for instance the books \cite{Bre1972,DK2000,tDie1987}), the case of linear representations is of special interest as it subtly mixes differential and algebraic techniques. Some results are therefore specific to this context. All these considerations have led us to write an article that is half way between a survey on the geometry of linear compact group actions, and an original research paper.

\textbf{Acknowledgments:} The third author warmly thanks Antoine Ducros for numerous discussions, in particular for pointing reference \cite{Ser1993} to him.

\section{Differential geometry background}
\label{sec:diffgeometry}

\subsection{Isotropy strata}

Let $K$ be a compact Lie group. The inclusion relation on the set of closed subgroups of $K$ induces a partial order on the set of their conjugacy classes, called \emph{containment relation}. It is defined as follows:

\begin{defn} \label{def:preceq}
  $[L_{1}]\preceq [L_{2}]$ if and only if $L_{1}$ is conjugate to a subgroup of $L_{2}$.
\end{defn}

To prove that this definition is meaningful, we recall the following classical lemma (in a slightly more general form that what is needed right now).

\begin{lem} \label{lem:finitecc}
  If $G$ is a Lie group with a finite number of connected components and $u \colon G \rightarrow G$ is an injective group morphism, then $u$ is an isomorphism.
\end{lem}

\begin{proof}
  Since $du_e$ is injective, it is an isomorphism of $\mathfrak{g}$. Hence $u$ is a local diffeomorphism near $e$, so $u$ induces an isomorphism of the connected component $G^e$ of $e$. From the diagram
  \[
    \xymatrix{1 \ar[r] & G^e \ar[r] \ar[d]^-{\wr} & G \ar[r] \ar[d]^-{u}  & G/G^e \ar[r] \ar[d] & 1 \\
    1 \ar[r] & G^e \ar[r]  &G \ar[r]  & G/G^e \ar[r]  & 1
    }
  \]
  the action of $u$ on $G/G^e$ is also injective. Since $G/G^e$ identifies with the set of connected components of $G$, it is finite by hypothesis, so the action of $u$ on $G/G^e$ is an isomorphism. Hence $u$ is an isomorphism.
\end{proof}

It follows from this lemma that if $gLg^{-1} \subset L$, then $g$ normalizes $L$. This implies that the containment relation $\preceq$ is antisymmetric.

\begin{rem}
  In the compact case, the fact that $g$ normalizes a closed subgroup $L$ if and only if $gLg^{-1} \subset L$ holds without assuming that $K$ is a Lie group (see \cite[prop. 3.7]{DK2000}).
\end{rem}

Let $\rho: K \to \GL(W)$ be a continuous linear representation of $K$ on a finite dimensional real vector space $W$. It is well-known that $\rho$ is real-analytic. In particular, we can apply all the tools of differential topology concerning smooth actions of Lie groups on manifolds.

We write $k.w:=\rho(k)(w)$ to lighten the notation. For any vector $w$ in $W$, we denote by $K \cdot w$ its orbit under $K$, and by $K_{w}$ the stabilizer of $w$ by the action, it is a Lie subgroup of $K$. The orbit $K \cdot w$ is a smooth compact submanifold of $W$, which is diffeomorphic to the homogeneous space $K/K_{w}$.

\begin{defn} \label{def:orbittype}
  The orbit types of $\rho$ form the set of conjugacy classes of stabilizers $[K_w]$ where $w$ runs through vectors of $W$. It is a partially ordered set for the containment relation.
\end{defn}

\begin{defn} \label{def:strate}
  Let $[L]$ be an orbit type. The fixed locus $W^L$ is the vector subspace of $W$ of elements fixed by $L$. The stratum $\Sigma_{[L]}$ is the set of vectors whose stabilizer is conjugate to $L$. The open fixed locus $W^{\langle L \rangle}$ is $\Sigma_{[L]} \cap W^L$.
\end{defn}

\begin{rem}
The fixed locus $W^{L}$ is defined for arbitrary closed subgroups of $K$ and not only for isotropy subgroups but we will avoid to do so because it can lead to confusions. For instance, some symmetries provided in the classification of \cite{FV1996} are not orbit types.
\end{rem}

\subsection{Slices} \label{sec:slices}

We recall here the general definition of slices, but in our setting not all properties are necessary (see remark \ref{rem:slice}).

\begin{defn}\label{def:slice}
  For any $w$ in $W$, a local slice of the $K$-action at $w$ is an embedded disc $S \subset W$ passing through $w$ such that:
  \begin{enumerate}
    \item[(i)] $S$ is transverse to the orbit $K \cdot w$.
    \item[(ii)] $S$ is stable under $K_{w}$.
    \item[(iii)] If $s_{1}, s_{2}$ are two points in $S$ and if there exists $k$ in $K$ with $k.s_{1}=s_{2}$, then $k$ belongs to $K_{w}$. In particular, $K_{s} \subset K_{w}$ for each point $s \in S$.
    \item[(iv)] $K \cdot S$ is an open neighborhood of the orbit $K \cdot w$.
  \end{enumerate}
\end{defn}

\begin{rem}\label{rem:slice}
  The following additional properties hold :

  -- Given a slice $S$, the open $K$-stable neighborhood $K \cdot S$ is diffeomorphic to the quotient $K \times_{K_{w}} S$, and after linearizing the action of $K_{w}$ on $S$, $K \cdot S$ is locally diffeomorphic to $K \times_{K_{w}} N_{w}$ where $N_{w}$ is the normal space of the orbit at $w$ (viewed as a representation of the isotropy group $K_{w}$).

  -- Slices exist in general for proper actions of Lie groups on manifolds~\cite[thm. 2.3.1]{DK2000}, and in particular for actions of compact Lie groups. However, for linear representations of compact Lie groups, it is possible to produce slices explicitly.

  -- For smooth actions of compact Lie groups on compact manifolds, the theory of slices combined with induction on the dimension yields the finiteness of orbit types (see \cite[Ch VII]{Bor1960} and~\cite[thm 5.11]{tDie1987}).

  -- This implies the finiteness of orbit types for a linear representation of a compact Lie group. To see it, it suffices to take an invariant inner product and then to consider the induced action on the unit sphere of $W$ .

  -- In our setting, where $K$ is compact, (iii) and (iv) are automatic after shrinking $S$.

  -- Property (iv) can be made slightly stronger: up to shrinking $S$, we can even assume that the orbit map from $K \times S$ to $V$ is a submersion, in particular it is an open map.
\end{rem}

\begin{cor}\label{cor:slices2}
  Let $w$ in $W$. Then there exists a neighborhood $U$ of $w$ such that for all $x$ in $U$, $[K_{x}]\preceq [K_{w}]$.
\end{cor}

\begin{proof}
  Let $S$ be a slice at $w$. If $x$ is in $K \cdot S$ then $w$ is in the orbit of a point $y$ in $S$. It follows that $K_{y} \subset K_{w}$, so $[K_x]=[K_{y}] \preceq [K_{w}]$.
\end{proof}

\begin{cor}
  For any orbit type $[L]$, $W^{\langle L \rangle}$ is open in $W^L$.
\end{cor}

\begin{proof}
  Let $x$ be in $W^{\langle L \rangle} $. Thanks to corollary~\ref{cor:slices2} there exists a neighborhood $U$ of $x$ such that for $y$ in $U$, $[K_{y}] \preceq [L]$. If $y$ in $U \cap W^{L}$, $[L] \preceq [K_{y}] \preceq [L]$ so $[K_{y}]=[L]$, which implies that $K_{y}=L$.
\end{proof}

\begin{defn}
  A vector $w \in W$ is called \emph{principal} if there exists a neighborhood $U$ of $w$ such that for all $x$ in $U$, $[K_x]=[K_{w}]$. Equivalently, $w$ is principal if it has locally minimal isotropy.
\end{defn}

\begin{rem} By definition, the set of principal points is open. Our definition of principal points is not totally standard. Some authors, like in \cite[\S 6]{AS1983} have adopted another definition: a vector $w \in W$ is principal if there exists a neighborhood $U$ of $w$ such that  for all $x$ in $U$, $[K_{w}] \preceq [K_{x}]$. For linear representations of compact Lie groups, the two definitions are equivalent.
\end{rem}

Let us denote by $p$ the canonical projection from $W$ to $W/K$. The principal orbit type theorem can be stated in the case of linear representations as follows:

\begin{thm}[{Principal orbit type theorem, see \cite[thm 3.1]{Bre1972}, \cite[\S VI]{AS1983}, \cite[thm 2.8.5]{DK2000}}]\label{thm:principal}
  For any open subset $\Omega$ of $W$ such that $p(\Omega)$ is connected, the set $U_{\Omega}$ of principal points  in $\Omega$ is open, dense in $\Omega$, and $p(U_{\Omega})$ is connected.
\end{thm}

\begin{cor}
  Orbit types of $\rho$ form a finite partially ordered set admitting a minimal element. The stratum corresponding to minimal elements is the dense open set of principal points in $W$.
\end{cor}

\subsection{Closures of isotropy strata}

The key point to study closure of isotropy strata is to prove that corollary \ref{cor:slices2} holds in the Zariski topology. The statement below will be the object of proposition \ref{prop:delicat} later on.

\begin{prop}\label{prop:dense} cf. \cite[lem. 5.5 (2)]{Sch1980}
  Let $L$ be an isotropy subgroup. Then $W^{\langle L \rangle}$ is a nonempty Zariski open subset of $W^{L}$. In particular it is open and dense for the usual topology.
\end{prop}

\begin{prop} \label{prop:closedstratum}
  Given an isotropy type $[L]$, the closure of $\Sigma_{[L]}$ is
  \[
    \overline{\Sigma}_{[L]}=K \cdot W^{L} = \{w \in W, [L] \preceq [K_w]\}= \displaystyle{\bigcup_{\substack{
    {[L'] \in \textrm{or}(\rho)} \\ [L] \preceq [L']}}\Sigma_{[L']}}.
  \]
\end{prop}

\begin{proof}
  The first equality follows by combining the fact that $\Sigma_{[L]} = K \cdot W^{\langle L \rangle}$ with proposition \ref{prop:dense}. The two other ones are straightforward.
\end{proof}

\begin{rem}\label{rem:pourboris}
  The equality $\overline{\Sigma}_{[L]}= \{w \in W, [L] \preceq [K_w]\}$ given by Proposition \ref{prop:closedstratum} is sometimes taken for granted in the literature (like for instance in \cite[\S 3.2]{AKP2013}). Although the right-hand side is closed thanks to corollary \ref{cor:slices2}, it is in general strictly bigger than the closure of the isotropy stratum $\Sigma_{[L]}$ for general actions of compact Lie groups, as shown in the example below.
\end{rem}

\begin{exmp} \label{ex:pourboris2}
  Let $R$ be the diagonal matrix $\mathrm{diag}(1, 1, \mathbf{i}, - \mathbf{i}) $ acting on the complex projective space $\mathbb{P}^3(\CC)$.  It defines an action of $\mathbb{Z}/4\mathbb{Z}$ on $\mathbb{P}^3(\CC)$. If $[x, y, z, t]$ are the homogeneous coordinates on $\mathbb{P}^3(\CC)$, then :

  -- Principals points are points outside the two projective lines $\Delta_1=\{x=y=0\}$ and $\Delta_2=\{z=t=0\}$.

  -- Points with isotropy $\mathbb{Z}/2\mathbb{Z}$ are points on $\Delta_1 \setminus \{A, B\}$ where $A=[0, 0, 1, 0]$ and $B=[0, 0, 0, 1]$.

  -- Points with maximal isotropy $\mathbb{Z}/4\mathbb{Z}$ are points in $\Delta_2$, together with $A$ and $B$.

  We see in this example that taking the closure of the $\mathbb{Z}/2\mathbb{Z}$-isotropy stratum adds the two points $A$ and $B$ that have maximal isotropy, but misses the line $\Delta_2$. This can be seen on the figure below, where the blue (resp. red) color corresponds to the $\mathbb{Z}/2\mathbb{Z}$ (resp. $\mathbb{Z}/4\mathbb{Z}$) isotropy strata.
  \par
  \begin{center}
    \begin{tikzpicture}
      \draw[thick, blue] (0,0) -- (5,1) ;
      \draw[thick, red] (0,-1) -- (5,0) ;
      \draw (5,1) node[right]{$\Delta_1$};
      \draw (5,0) node[right]{$\Delta_2$};
      \draw (2,0.5) node[above]{$A$};
      \draw (3,0.7) node[above]{$B$};
      \draw[red] (2,0.4) node{$\bullet$};
      \draw[red] (3,0.6) node{$\bullet$};
    \end{tikzpicture}
  \end{center}
\end{exmp}

For a compact Lie group representation $(W,\rho)$, the partition into (nonempty) isotropy strata
\begin{equation*}
  W = \Sigma_{[L_{0}]} \cup \dotsb \cup \Sigma_{[L_{n}]}
\end{equation*}
is called its \emph{isotropy stratification} or \emph{orbit type stratification}. It can be shown that it is a real stratification, and even a Whitney stratification (see~\cite[\S 2.7 and thm. 2.7.4]{DK2000}).

\section{Algebraic background}
\label{subsec:invariants}

\subsection{Reductive groups and invariants}

In the material we will present, we will constantly deal with two types of groups:

-- Compact Lie groups: they admit automatically a faithful representation~\cite[thm. III.4.1]{BD1985}, and thanks to Weyl's unitary trick, this representation can be chosen unitary. Besides, every continuous finite-dimensional representation of a compact Lie group is fully irreducible, i.e. splits as a direct sum of irreducible representations.

-- Complex reductive groups : they are complex Lie groups admitting a faithful complex analytic representation and such that every finite-dimensional analytic representation splits as a direct sum of irreducible representations.

Complex reductive groups are complexifications of compact Lie groups~\cite[thm. 4.31]{Lee2002}. If $G \subset \mathrm{GL}_n(\CC)$ is reductive, then $G$ is automatically an algebraic subgroup of $\mathrm{GL}(n, \CC)$~\cite[thm. 5.11]{Lee2002} and any representation of $G$ is rational \cite[\S 5.5, cor. 1]{Ser1993}.

Let $K$ (resp. $G$) be a real (resp.\ complex) Lie group. The linear action of $K$ (resp. $G$) on a real (resp. complex) vector space $W$ (resp. $V$) extends naturally to the polynomial algebra $\mathbb{R}[W]$ (resp. $\mathbb{C}[V]$) via the formula $(g.P)(x) := P(g^{-1}. x)$. The set of all polynomials that are invariants under the action of $K$ (resp. $G$) is a subalgebra of $\mathbb{R}[W]$ (resp.  $\mathbb{C}[V]$) denoted by $\mathbb{R}[W]^K$ (resp. $\mathbb{C}[V]^{G}$) and called the \emph{invariant algebra} of $W$ (resp. $V$).
The foundational result of invariant theory, due initially to Hilbert~\cite{Hil1993} in the case of the action of $\mathrm{SL}(n; \CC)$ on complex binary forms, runs as follows:
\begin{thm}\label{thm:hilbert}~\cite[Theorem X.5.6]{Hel2001},~\cite[Theorem 6.3.1]{LPot1997}
  Let $K$ \emph{(}resp. $G$\emph{)} be a compact \emph{(}resp.\ complex reductive\emph{)} Lie group, let $W$ \emph{(}resp. $V$\emph{)} be a finite dimensional real \emph{(}resp. complex\emph{)} vector space, and let $\rho\colon K \to \GL(W)$ \emph{(}resp. $\rho\colon G \to \GL(V)$\emph{)} be a continuous \emph{(}resp.\ analytic\emph{)} representation of $K$ \emph{(}resp. $G$\emph{)}. Then, the invariant algebra $\mathbb{R}[W]^{K}$ \emph{(}resp. $\mathbb{C}[V]^{G}$\emph{)} is finitely generated. This means that there exists a finite set of invariant polynomials $J_{1},\dots ,J_{N}$ such that
  $\mathbb{R}[W]^{K}= \mathbb{R}[J_{1},\dotsc,J_{N}]$ \emph{(}resp. $\mathbb{C}[V]^{G}=\mathbb{C}[J_{1},\dotsc, J_{N}]$\emph{)}.
\end{thm}

\begin{rem}\label{rem:discussion}
  Although this theorem is stated most of the time for complex reductive groups (or even reductive groups over an algebraically closed field of characteristic zero), the proof in the compact case works in the same way since it relies only on the Noetherianity of $\mathbb{R}[W]$ and the existence of a Reynolds operator.
\end{rem}
Recall that an integral ring is normal if it is integrally closed in its fraction field. Using the preceding notation, we have the following result:
\begin{lem} \label{lem:normal}
The ring $\mathbb{R}[W]^K$ \emph{(}resp. $\mathbb{C}[V]^G$\emph{)} is a normal ring.
\end{lem}

\begin{proof}
  We have $\mathbb{R}[W]^K = \mathrm{Frac}(\mathbb{R}[W]^K) \cap \mathbb{R}[W]$. If an element $x$ in $\mathrm{Frac}(\mathbb{R}[W]^K)$ is integral over $\mathbb{R}[W]^K$, it is also integral over $\mathbb{R}[W]$ and since $\mathbb{R}[W]$ is normal, $x$ belongs to $\mathbb{R}[W]$, hence to $\mathbb{R}[W]^K$. The same proof works in the complex case since $\mathbb{C}[V]$ is normal.
\end{proof}

It is clear that any invariant is constant on $G$-orbits. The geometry of orbits can be understood via the invariants, but the situation is different for the real and the complex cases.

\textbf{Real case $(K, W)$:}

-- The $K$-orbits are compact, smooth, and are algebraic subsets of $W$.

-- The invariants separate the $K$-orbits. In other terms, given two vectors $w_{1}, w_{2} \in W$ belonging to different $K$-orbits, it is always possible to find a function $J \in \RR[W]^{K}$ such that $J(w_{1})\neq J(w_{2})$ (see~\cite[Appendix C]{AS1983}).

-- The orbit space $W/K$ can be described as a semialgebraic subset $S$ of $\RR^N$. Indeed, if $\set{J_{1},\dotsc, J_{N}}$ denotes a generating set for $\RR[W]^{K}$, then the mapping
\begin{equation*}
  p \colon w \mapsto \big(J_{1}(w), J_{2}(w), \dotsc, J_{N}(w)\big)
\end{equation*}
induces an homeomorphism between $W/K$ and $p(W)\subset \RR^{N}$ which is a semialgebraic subset of $\RR^N$ that can be described explicitly using the Gram matrix (we refer the reader to \cite{PS1985} for more details).

\textbf{Complex case $(G, V)$:}

The properties below are classical in geometric invariant theory, we refer the reader for instance to \cite[Part I \S 6]{LePotier_1997}.

-- The $G$-orbits are locally closed.

-- Two Zariski-closed $G$-stable sets of $V$ can be separated by invariants.

-- Each $G$-orbit is adherent to a unique closed $G$-orbit.

-- The complex scheme $V//G:=\mathrm{spec}\, (\CC[V]^{G})$ parameterizes closed $G$-orbits.

-- The map $q \colon V \rightarrow V // G$ maps $G$-invariant Zariski closed sets to Zariski closed sets.

\subsection{Luna's slice theorem}
\label{subsec:algebraic-case}

Let $G$ be a complex reductive group acting algebraically on a complex affine algebraic variety (the theory works for any closed field of characteristic zero, but we will use it essentially in the complex case). The aim of this section is to give a quick introduction to Luna's slice theorem, that produces a local model for the action near a closed orbit of $G$. This theorem is the algebraic counterpart of the existence of slices for proper actions of Lie groups on manifolds.

The most simple example happens when the action is free. If we pursue the analogy with differential geometry, we would expect a local trivialization near the orbit. However, the following example shows that it is not possible to expect such a result in the Zariski topology.

\begin{exmp}
  Let us consider the group of $n$-th roots of unity, acting naturally on $\CC^{*}$. In this case, the space of orbits is also isomorphic by $\CC^{*}$, the quotient map $\pi$ being given by $z \mapsto z^n$. We see that it is impossible to trivialize $\pi$ in the Zariski topology on the base. Indeed, a Zariski open subset of $\CC^{*}$ is obtained by removing a finite number of points, and the projection remains nontrivial on any such open subset.
\end{exmp}

The problem comes from the fact that a smooth and surjective morphism between algebraic varieties does not always have a section in the Zariski topology. However, it has a section in the etale topology~\cite[17.16.3 (ii)]{Gro1967}. Concretely this means the following: if $\varphi \colon X \rightarrow Y$ is smooth and surjective, then for any $y$ in $Y$ there is a neighborhood $U_{y}$ of $y$ and an etale morphism $\varphi \colon V \rightarrow U_{y}$ from an algebraic variety $V$ to $U_y$ such that the pull-back morphism $\widetilde{f} \colon V \times_{U_{y}} X \rightarrow V$ has a section. For the interested reader, let us mention that the problem of trivializing principal $G$-bundles in the Zariski topology was studied in depth by Grothendieck in~\cite{Gro1958}: in fact he proved that for a given algebraic group $G$, then all etale locally trivial principal $G$-bundles are Zariski locally trivial if and only if $G$ is affine, connected, and without torsion~\cite[Theorem 3]{Gro1958}.

For any affine complex variety $X=\mathrm{Spec}\,A$ ($A$ being a commutative complex algebra of finite type) endowed with an action of a reductive group $G$, we denote by $X \,/\!/\,G$ the categorical quotient $\mathrm{Spec}\,A^G$, the points of $X \,/\!/\,G$ parametrize closed $G$-orbits of $X$. We say that a subset $S$ of $X$ is $G$-saturated if it is saturated with respect to the quotient map $X \rightarrow X \,/\!/\,G$. This means that any point in $X$ whose $G$-orbit is adherent to a point of $S$ lies in $S$.

\begin{thm}[Luna's slice theorem \cite{Lun1973}] \label{thm:lunaslice}
  Let $G$ be a complex reductive group acting on a complex affine algebraic variety $X$. Let $x$ be a point of $X$ and assume that the orbit $G \cdot x$ is closed in $X$. Then there exists an \textit{etale slice} at $x$, that is an affine subvariety $V$ of $X$ passing through $x$ satisfying the following properties:
  \begin{enumerate}
    \item[--] $V$ is $G_{x}$-stable
    \item[--] The natural $G$-morphism $\varphi \colon G \times_{G_{x}} V \rightarrow X$ is etale, and its image $U$ is an affine and $G$-saturated neighborhood of the orbit $G \cdot x$.
    \item[--] The natural map $\overline{\varphi} \colon V \,/\!/\, G_x \rightarrow U \,/\!/\, G$ is etale.
    \item[--] The diagram
          \[
            \xymatrix{G \times_{G_x} V \ar[r]^-{\varphi} \ar[d]^-{\sigma} & U \ar[d]^{\pi} \\ V  \,/\!/\, G_x \ar[r]^-{\overline{\varphi}}& U  \,/\!/\, G }
          \]
          is cartesian.
  \end{enumerate}
\end{thm}

\begin{cor} \label{cor:closed}
  With the notations of theorem \ref{thm:lunaslice}, if $v$ is in $V$ and if $G_x \cdot v$ is closed in $V$, then $G \cdot v$ is closed in $X$.
\end{cor}

\begin{proof}
  Assume that $G_x \cdot v$ is closed in $V$, and let $z$ in $U$ in the closure of $G \cdot v$. Then we have the diagram
  \[
    \xymatrix{
    & z \ar[d]^-{\pi} \\
    \sigma(v) \ar[r]^{\overline{\varphi}} & \pi(z)
    }
  \]
  which means that there exists $(g, v')$ in $G \times V$ such that $g.v'=z$ and $\sigma(v')=\sigma(v)$. Since $G_x \cdot v$ is closed, $v'$ lies in the $G_x$-orbit of $v$, so $v'$ lies in $G \cdot v$ and so does $z$.
\end{proof}

A corollary of the slice theorem is the algebraic version of corollary~\ref{cor:slices2}. Since complex algebraic groups have a finite number of connected components, we can use for them the containment relation without problem as we did for the compact case thanks to lemma \ref{lem:finitecc}.

\begin{cor}~\cite[Remark $4^{\circ}$ pp. 98]{Lun1973}\label{cor:yolo1}
  Let $G$ be a reductive group acting on an affine algebraic variety $X$. Let $x$ be a point of $X$ and assume that the orbit $G \cdot x$ is closed in $X$. Then there exists a $G$-invariant saturated Zariski open neighborhood $U$ of $x$ such that for any $y$ in $U$, $[G_{y}] \preceq [G_x]$. Besides, if for some $y$ in $U$ we have $[G_{y}]=[G_x]$, then $G \cdot y$ is closed.
\end{cor}

\begin{proof}
  The first point is straightforward since the stabilizers of points in $G \times_{G_x} V$ are subgroups of $G_x$. For the second point, if $[G_y]=[G_x]$, we can assume that $y$ lies in $V$ so $G_y=G_x$. It follows that $G_x \cdot y = \{y\}$, so by corollary \ref{cor:closed}, $G \cdot y$ is closed.
\end{proof}

The closedness of the orbit is crucial in Luna's slice theorem, as shown by the example below, due to Richardson:

\begin{exmp}\cite[Remark $4^{\circ}$ page 98]{Lun1973}\label{ex:sournois}
  Let us consider the natural action of $\SL(2;\CC)$ on the set $V_{3}$ of cubic binary forms, \textit{i.e.}, homogeneous complex polynomials of degree $3$ in two complex variables. There is a single invariant $\Delta$, which is the discriminant.

  \begin{table}[ht]
    \begin{tabular}{cccc}
      \toprule
      \textbf{Description}        & \textbf{Stabilizer}                  & $\Delta$ & \textbf{Type of orbit}     \\
      \midrule
      \textrm{three simple roots} & $\mathbb{Z}/3\mathbb{Z}$             & $\neq 0$ & \textrm{closed}            \\
      0                           & $\SL(2;\CC)$                         & $0$      & \textrm{closed}            \\
      \textrm{one double root}    & $\{0\} $                             & $0$      & \textrm{adherent to}\,\, 0 \\
      \textrm{one triple root}    & $\CC \rtimes \mathbb{Z}/3\mathbb{Z}$ & $0$      & \textrm{adherent to}\,\, 0 \\
      \bottomrule
    \end{tabular}
  \end{table}

  The set $U$ of binary forms which have three distinct roots is a Zariski open subset of $V_{3}$ defined by the nonvanishing of the discriminant. It is dense, and all points in $U$ have isotropy $\ZZ / 3 \ZZ$. However, any orbit of a cubic form with a double root has trivial isotropy. Hence the generic orbit has isotropy $\mathbb{Z}/3\mathbb{Z}$, but some other orbits have trivial stabilizer.
\end{exmp}

\subsection{Isotropy strata in the complex case}

Let $\rho \colon G \rightarrow \mathrm{GL}(V)$ be a finite-dimensional representation of a complex reductive group.
The definition of orbit types takes in consideration only the stabilizers of \textit{closed} orbits.
\begin{defn}
  The orbit types of $\rho$ is the set of conjugacy classes of stabilizers $[G_v]$ where $v$ runs through vectors of $V$ with closed orbit.
\end{defn}

Recall that thanks to lemma \ref{lem:finitecc}, the containment relation is still well-defined on conjugacy classes of closed algebraic subgroups of $G$. It follows from Luna's theorem and the standard induction argument that orbit types of $\rho$ is a finite partially ordered set.

For any $v$ in $V$, there is an associated isotropy class $[G_v^*]$ which is the conjugacy class of the isotropy group of any point in the unique closed orbit adherent to $G \cdot v$. Thanks to corollary \ref{cor:yolo1}, we always have $[G_v] \preceq [G_v^*]$.

\begin{prop}[Matsushima's criterion, cf {\cite[\S I.2]{Lun1973}}]
  If $[H]$ is an orbit type of $\rho$, then $H$ is reductive.
\end{prop}

We can again attach to $\rho$ a finite partial ordered set of orbit types, they consist of reductive subgroups of $G$. This partially ordered set is finite and has a unique minimal element.

\begin{rem}
  If $[H]$ is the minimal orbit type, the set of vectors $v$ in $V$ such that $[G_v^*]=[H]$ is a nonempty Zariski open set of $V$. It is even possible to prove that there exists a nonempty Zariski open subset of $V$ such that all stabilizers are conjugate, but this generic isotropy can be nonreductive. This result is due to Richardson (see \cite{Richardson1972}, \cite[corollary 8]{Lun1973}).
\end{rem}

\begin{defn}
  Let $[H]$ be an orbit type.

  -- The open stratum $\Omega_{[H]}$ is the set of vectors $v$ in $V$ such that $[G_v^*]=[H]$.

  -- The closed stratum ${\Lambda}_{[H]}$ is the set of vectors $v$ in $V$ such that $[H] \preceq [G_v^*]$.

  -- The open fixed locus $V^{\langle H \rangle}$ is $V^H \cap \Omega_{[H]}$.
\end{defn}

\begin{prop}\cite[lem. 5.5 (2)]{Sch1980} \label{prop:zdens}
  The open fixed locus $V^{\langle H \rangle}$ consists of all vectors in $V$ whose orbit is closed and whose stabilizer is $H$, it is a nonempty Zariski open subset of $V^H$.
\end{prop}

\begin{proof}
  Let $v$ in $V^{\langle H \rangle}$. Then $[H] \preceq [G_v] \preceq [G_v^*] = [H]$. Hence there is equality everywhere, so $G \cdot v$ is closed thanks to corollary \ref{cor:yolo1}, and $G_v=H$. The converse inclusion is obvious. To prove that $V^{\langle H \rangle}$ is Zariski open in $V^H$, let $v$ in $V^{\langle H \rangle}$. Then there exists a Zariski open subset $U$ of $V^H$ containing $v$ such that for any $x$ in $U$, $[G_x] \preceq [G_v]=[H]$. Hence, for $x$ in $U \cap V^H$, $[H] \preceq [G_x] \preceq [G_v]=[H]$ and we apply again corollary \ref{cor:yolo1} to deduce that $U \cap V^H$ is included in $V^{\langle H \rangle}$, whence the result.
\end{proof}

\begin{prop} \label{prop:lunaclosed}
  For any orbit type $[H]$ of $\rho$, ${\Lambda}_{[H]}$ is a Zariski closed subset of $V$.
\end{prop}

\begin{proof}
  If $v$ is not in ${\Lambda}_{[H]}$, let $G \cdot x$ be the unique closed orbit adherent to $G \cdot v$. Let $U$ be a $G$-invariant neighborhood of $G \cdot x$ given by corollary \ref{cor:yolo1}. Then $U$ is a Zariski neighborhood of $v$ and for any $y$ in $U$, $[G_y] \preceq [G_x]$. Then it is impossible to have $[H] \preceq [G_y]$ --otherwise we would have $[H] \preceq [G_x]=[G_v^*]$-- so $U$ is a Zariski neighborhood of $v$ in the complement of $\overline{\Sigma}_{[H]}$.
\end{proof}

Let $Z=V // G = \mathrm{spec}(\mathbb{C}[V]^G)$ and $q \colon V \rightarrow Z$ the natural projection. Then:

\begin{prop} \label{prop:direct}
  Given an isotropy type $[H]$ of $\rho$, let $Z_{[H]}=q(\Omega_{[H]})$.
  \begin{enumerate}
    \item[(i)] $Z_{[H]}=q(V^{\langle H \rangle})$.
    \item[(ii)] $\overline{Z}_{[H]}=q(V^H)=q(\Lambda_{[H]})= \displaystyle{\bigcup_{\substack{
            {[H'] \in \textrm{or}(\rho)} \\ [H] \preceq [H']}}Z_{[H']}}$.
    \item[(iii)] The $(Z_{[H]})_{[H] \in \mathrm{or}(\rho)}$.
          define a stratification of $Z$ by locally closed sets.
    \item[(iv)] $\Omega_{[H]}=q^{-1}(Z_{[H]})$.
    \item[(v)] $\Lambda_{[H]}=q^{-1}(\overline{Z}_{[H]})$.
  \end{enumerate}
\end{prop}

\begin{proof}
  There are obvious inclusions since $V^{\langle H \rangle}$ (resp. $V^H$) is included in $\Omega_{[H]}$ (resp. $\Lambda_{[H]})$. For the converse inclusions, if $v$ is in $\Omega_{[H]}$, by definition, there exists $v^*$ in $V$ such that $G \cdot v^*$ is closed, $q(v)=q(v^*)$, and $[G \cdot v^*]=[H]$. This means that $G_v^*$ equals $gHg^{-1}$ so $G_{g.v*}=H$ which implies that $g.v^*$ belongs to $V^{\langle H \rangle}$. This gives (i).

  For (ii), we are in the same situation except that $G_v^*$ contains $gHg^{-1}$, so $g.v^*$  belongs to $V^H$. This gives the inclusion of $q(\Lambda_{[H]})$ in $q(V^H)$. The last queality of (ii) is obvious, it remains to prove the first one. Since $\Lambda_{[H]}$ is Zariski closed by proposition \ref{prop:lunaclosed} and contains $Z_{[H]}$, it also contains $\overline{Z}_{[H]}$. Conversely, $q(V^H)=q\left(\overline{V^{\langle H \rangle}}\right) \subset \overline{q(V^{\langle H \rangle})}=\overline{Z}_{[H]}$.

  For (iii) we write that $\overline{Z}_{[H]} \setminus Z_{[H]} = \displaystyle{\bigcup_{\substack{
    {[H'] \in \textrm{or}(\rho)} \\ [H] \preceq [H'] \\ [H'] \neq [H]}} \overline{Z}_{[H']}}$ is closed in $\overline{Z}_{[H]}$.

  Properties (iv) and (v) are straightforward.
\end{proof}

\begin{rem}
  There is another natural closed stratum attached to $[H]$, which is smaller than $\Lambda_{[H]}$ : it is the Zariski closure of points $v$ with closed orbits such that $[G_v]=[H]$. This set is $\overline{G . V^{\langle H \rangle}}$, which is also equal to $\overline{G . V^{H}}$. It has the advantage to be always irreducible, contrarily to $\Lambda_H$, as shown in the next example.
\end{rem}

\begin{exmp} \label{baby}
  Let us consider the representation of $\mathbb{C}^{\times}$ on $\mathbb{C}^2$ given by $\lambda . (z, w)=(\lambda z, \lambda^{-1} w)$. There are two complex orbit types : $[0]$ and $[\mathbb{C}^{\times}]$. All closed orbits apart $0$ are of the form $\{zw=c\}$ for $c \neq 0$, and there are two non-closed orbits which are $\mathbb{C}^{\times} \times \{0\}$ and $\{0\} \times \mathbb{C}^{\times}$. Then $(\mathbb{C}^2)^{\mathbb{C}^{\times}}=\{0\}$ but ${\Lambda}_{[\mathbb{C}^{\times}]}=\{zw=0\}$.
\end{exmp}

The main feature of the stratum $\overline{G.V^{H}}$ has been provided by Luna:

\begin{prop} \cite[\S 2, prop.]{Lun1975} \label{prop:normal}
  If $[H]$ is a complex orbit type, the natural morphism $V^H // N_G(H) \rightarrow \overline{G . V^{H}}//G $ identifies with the normalization of $\overline{G.V^{H}}//G$.
\end{prop}

We won't go into the details of the proof which is quite involved. Let us just highlight two points :

-- By lemma \ref{lem:normal}, $V^H // N_G(H)$ is a normal variety.

-- If we look set-theoretically, the morphism $V^{\langle H \rangle} / N_G(H) \rightarrow (G.V^{\langle H \rangle})/G$ is a bijection.

\section{Complexification}
\label{sec:complexification}

\subsection{Compact Lie groups}

The complexification of a real Lie group is formally defined as the solution of a universal problem, which always exists and is unique up to a complex analytic isomorphism \cite[Chapter 3]{BB1958}. If $K$ is compact, it is possible to define $K^{\CC}$ as follows \cite[III.8]{BD1985}: let $\mathcal{A}$ be the algebra of representative functions on $K$, that is functions that generate a finite-dimensional representation inside $\mathcal{C}^0(K, \RR)$. Tannaka-Krein duality~\cite[III.7]{BD1985} guarantees that $K$ identifies with real characters of $\mathcal{A}$, that is every character of $\mathcal{A}$ is of the form $f \rightarrow f(k)$ for $k$ in $K$. Then $K^{\CC}$ is defined as the complex characters of $\mathcal{A}^{\CC}$.

\begin{exmp}
  If $K=\UU(1)$, $\mathcal{A}$ is the algebra of trigonometric polynomials $\RR[\cos(\theta), \sin(\theta)]$. Its complexification is the algebra $\CC[\cos(\theta), \sin(\theta)]$. Given a complex trigonometric polynomial $P(\cos \theta, \sin \theta)$, we can associate the Laurent Polynomial $P\left(\frac{z+1/z}{2}, \frac{z-1/z}{2\mathbf{i}} \right)$. In this way we see that
  \begin{equation*}
    \CC[\cos(\theta), \sin(\theta)] \simeq \CC[z, 1/z],
  \end{equation*}
  since every Laurent polynomial is uniquely determined by its restriction on $\UU(1)$. Hence, the complex characters of $\CC[\cos(\theta), \sin(\theta)]$ are exactly the points of $\CC^{\times}=\mathrm{Spec}\, \mathbb{C}[z, 1/z]$.
\end{exmp}

There is an explicit way to describe this complexification using the polar decomposition, which is quite useful to understand more precisely the geometry of $K^{\CC}$. To achieve this, we use the fact that every compact Lie group admits a faithful representation~\cite[Thm. III.4.1]{BD1985}, and thanks to Weyl's unitary trick, it admits a faithful unitary representation.
Let $\mathrm{P}(n)$ be the set of hermitian positive definite matrices. The product map
\[
  \mathrm{U}(n) \times \mathrm{P}(n) \xrightarrow{\sim} \mathrm{GL}(n; \CC)
\]
is a diffeomorphism. If $\iota$ denotes the Cartan involution $M \rightarrow (M^*)^{-1}$ of $\mathrm{GL}(n; \CC)$, then for any $M$ in $\mathrm{GL}(m; \CC)$ with polar decomposition $kh$, $h^2=\iota(M)^{-1}M$.
The map $Z \rightarrow e^{\mathbf{i}Z}$ from $\mathfrak{u}(n)$ to $\mathrm{P}(n)$ is a diffeomorphism, where $\mathfrak{u}(n)=\mathrm{Lie}(\mathrm{U}(n))$ is the Lie algebra of skew-hermitian matrices.

\begin{prop}\cite[prop. III.8.3]{BD1985}\label{prop:Chevalley-complexification}
  Let $K\subset \mathrm{U}(n)$ be a compact Lie group and let $\mathfrak{k}$ be its Lie algebra. Then
  \begin{equation*}
    K^{\CC}=\set{ke^{\mathbf{i}Z}~;~ k\in K \text{ and } Z \in \mathfrak{k}}.
  \end{equation*}
  In particular, $K^{\CC}$ is diffeomorphic to $K \times \mathfrak{k}$.
\end{prop}

\begin{rem}
  This result might be surprising at first glance because it is not clear at all that the right hand side is a group. Let us briefly explain by hand why it is the case. Given $k_{1}e^{\mathbf{i}Z_{1}}$ and $k_{2}e^{\mathbf{i}Z_{2}}$, where $k_{1},k_{2}\in K$ and $Z_{1},Z_{2}\in \mathfrak{k}$, we have
  \begin{equation*}
    e^{\mathbf{i}Z_{1}} k_{2} =k_{2} \left(k_{2}^{-1}e^{\mathbf{i}Z_{1}}k_2 \right) = k_{2} e^{\mathbf{i} \mathrm{Ad}(k_{2}^{-1})(Z_{1})}=k_{2} e^{\mathbf{i}Z_3},
  \end{equation*}
  where $Z_3 \in \mathfrak{k}$. Thus we get $k_{1} e^{\mathbf{i}Z_{1}} k_{2} e^{\mathbf{i}Z_{2}}=k_{1}k_{2} e^{\mathbf{i}Z_{2}}e^{\mathbf{i}Z_3}$.
  Writing $e^{\mathbf{i}Z_{2}}e^{\mathbf{i}Z_3}=ke^{\mathbf{i}Z}$  with $k \in \mathrm{U}(n)$ and $Z \in \mathfrak{u}(n)$, we have
  \begin{equation*}
    e^{2\mathbf{i}Z}=\iota(e^{\mathbf{i}Z_{2}}e^{\mathbf{i}Z_3})^{-1} e^{\mathbf{i}Z_{2}}e^{\mathbf{i}Z_3}=e^{\mathbf{i}Z_3} e^{2\mathbf{i}Z_{2}} e^{\mathbf{i}Z_3}.
  \end{equation*}
  Consider now the real analytic function
  \begin{equation*}
    \varphi \colon \mathfrak{k} \times \mathfrak{k} \to \mathfrak{u}(n), \qquad (Z_{2},Z_3) \mapsto \frac{1}{2\mathbf{i}} \log(e^{\mathbf{i}Z_3}e^{2\mathbf{i}Z_{2}}e^{\mathbf{i}Z_3}).
  \end{equation*}
  By applying the Baker-Campbell-Hausdorff formula two times, we get that $\varphi(Z_{2},Z_3)\in \mathfrak{k}^{\CC}$ for sufficiently small $Z_{2},Z_3$, and by real-analytic continuation, we conclude that for all $Z_{2},Z_3$ in $\mathfrak{k}$, $\varphi(Z_{2},Z_3)\in \mathfrak{k}^{\CC}\cap \mathfrak{u}(n)=\mathfrak{k}$. This proves that for all $Z_{2}, Z_3$ in $\mathfrak{k}$, $Z$ is in $\mathfrak{k}$.
  We can apply the same trick a second time and consider the real analytic map
  \begin{align*}
    \Psi \colon \mathfrak{k} \times \mathfrak{k} & \to \mathrm{U}(n)                                                               \\
    (Z_{2},Z_3)                                  & \mapsto e^{\mathbf{i}Z_{2}}e^{\mathbf{i}Z_3}e^{-\mathbf{i}\varphi(Z_{2}, Z_3)}.
  \end{align*}
  For $Z_{2}, Z_3$ close to $0$, $\log \Psi$ takes values in $\mathfrak{k}^{\CC} \cap \mathfrak{u}(n)=\mathfrak{k}$, so $\Psi$ takes values in $K$. As $K$ is a closed real analytic submanifold of $\mathrm{U}(n)$ we deduce using again  real-analytic continuation that for all $Z_{2}, Z_3$ in $\mathfrak{k}$, $\Psi(Z_{2}, Z_3) \in K$. Since $\Psi(Z_{2}, Z_3)=k$, we conclude that $k$ belongs to $K$ and we are done.
\end{rem}

\begin{rem}
  The group $K$ is a closed subgroup of $\mathrm{GL}(n; \CC)$, and it is straightforward to check that $\mathrm{Lie}(K^{\CC})=\mathfrak{k}+\mathbf{i}\mathfrak{k}=\mathfrak{k}^{\CC}$ since $\mathfrak{u}(n)$ is totally real in $\mathfrak{gl}(n; \CC)$. Hence $\mathrm{Lie}(K^{\CC})$ carries a natural complex structure for which the bracket is complex linear. It defines on $K^{\CC}$ a complex Lie group structure for which the exponential map is a local biholomorphism around $0$. Besides, $K$ is totally real in $K^{\CC}$.
\end{rem}

\begin{prop} \cite[prop. III.8.2]{BD1985} \label{prop:algebraic}
  The complexified group $K^{\CC}$ is an affine algebraic subgroup of $\mathrm{GL}(n; \CC)$.
\end{prop}

\begin{cor}\label{cor:totrel}
  If $K$ is a compact Lie group, then $K$ is analytically Zariski dense in $K^{\CC}$.
\end{cor}

\begin{proof}
  The group $K$ is a totally real analytic submanifold of $K^{\CC}$. Hence the Zariski closure of $K$ in $K^{\CC}$ is the union of connected components of $K^{\CC}$ that intersect $K$. By Proposition~\ref{prop:Chevalley-complexification}, this union is $K^{\CC}$ itself.
\end{proof}

\begin{cor}\label{cor:cx}
  Let $K$ be a compact Lie group, and assume that $K^{\CC}$ acts holomorphically on a (non-necessarily finite dimensional) complex vector space $V$. Then $V^{K^{\CC}}=V^{K}$.
\end{cor}

\begin{proof}
  The inclusion of $V^{K^{\CC}}$ in $V^K$ is obvious. In the other direction, for $v$ in $V^{K}$, let $u$ be a linear form on $V$ and let  $\phi \colon K^{\CC} \rightarrow \CC$ given by $\phi(k)=u(k \cdot v-v)$. Then $\phi$ is holomorphic and $\phi$ vanishes on $K$. Thanks to corollary~\ref{cor:totrel}, $\phi$ vanishes on $K^{\CC}$. This means that for all $k$ in $K^{\CC}$, all linear forms take the same values on $v$ and $k \cdot v$, so $k \cdot v = v$. Hence $v$ is fixed by $K^{\CC}$.
\end{proof}

Let us check that $K^{\CC}$ satisfies the universal property of complexification (which is more general than~\cite[prop.  III.8.6]{BD1985} which concerns only linear representations).

\begin{prop}[{\cite[\S 5.3, thm 3]{Ser1993}}]\label{def:universal-property}
  Let $K$ be a compact Lie group, let $K^{\CC}$ its complexification, and let $\phi \colon K \rightarrow K^{\CC}$ be the natural injection. For any complex Lie group $H$ and any morphism $f \colon K \rightarrow H$ of Lie groups, there exists a unique homomorphism $F:K^{\CC} \to H$ of complex Lie groups such that $f = F \circ \phi$.
  \begin{equation*}
    \xymatrix{
    K \ar[r]^-{f} \ar[d]_-{\phi}   & H \\
    K^{\CC} \ar[ur]_-{F} }
  \end{equation*}
\end{prop}

\begin{proof}
  We use the notation of proposition~\ref{prop:Chevalley-complexification}. Let $\sigma$ denotes the differential of $f$ at the origin and let $\sigma^{\CC} \colon \mathfrak{k}^{\CC} \rightarrow \mathfrak{h}$ be its complexification. We define a map $F \colon K^{\CC} \rightarrow H$ by $F(ke^{iZ})=f(k)e^{i\sigma(Z)}$. First we claim that $F$ is holomorphic. By Lie's third theorem (see~\cite[\S II.8, Thm 1]{Ser2006}), the morphism $\sigma^{\CC}$ can be uniquely integrated to a local holomorphic group morphism $\widetilde{F} \colon K^{\CC} \dashrightarrow H$ defined in a neighborhood of the identity, whose differential at the identity element is $\sigma^{\CC}$. Hence, $\widetilde{F}_{|K}$ integrates $\sigma$, so $\widetilde{F}_{|K}=f$. Hence, for $k$ close to the identity and $Z$ close to $0$, we get
  \begin{equation*}
    \widetilde{F}(ke^{iZ})=\widetilde{F}(k) \widetilde{F}(e^{iZ})=f(k)e^{\sigma^{\CC}(iZ)}=f(k) e^{i\sigma(Z)}=F(ke^{iZ}),
  \end{equation*}
  so $F=\widetilde{F}$ near the identity. It follows that $F$ is holomorphic near the identity. Since $F$ is real analytic, $F$ is holomorphic on $(K^{\CC})^e$. Since for $k$ in $K$ and $\ell$ in $K^{\CC}$, $F(k\ell)=f(k)F(\ell)$, it follows that $F$ is holomorphic on $K^{\CC}$. Lastly, $F_{|K}=f$ is a group morphism on $K$, so by corollary~\ref{cor:totrel}, applied two times, $F$ is a group morphism.
\end{proof}

\begin{prop} \label{prop:vaudaine}
  Let $K$ be a compact Lie group, let $G=K^{\CC}$, and
  let $L_1$ and $L_2$ be two closed subgroups of $K$. If $[L_1^{\CC}] \preceq [L_2^{\CC}]$, then $[L_1] \preceq [L_2]$.
\end{prop}

\begin{proof}
  By hypothesis, there is $g$ in $G$ such that $gL_1^{\CC}g^{-1} \subset L_2^{\CC}$. The group $gL_1g^{-1}$ is a compact subgroup of $L_2^{\CC}$, so it is included in a maximal compact subgroup of $L_2^{\CC}$. Since these maximal compact subgroups are all conjugate, we can assume that $gL_1g^{-1}$ is included in $L_2$. For any $\ell_1$ in $L_1$ let $\ell_2=g \ell_1 g^{-1}$, and write $g^{-1}=ke^{\mathbf{i}Z}$ where $Z$ is in $\mathfrak{k}$. Then
  $e^{-\mathbf{i}Z}k^{-1} \ell_1 k e^{\mathbf{i}Z}=\ell_2$, so
  if we put $k_1=k^{-1} \ell_1 k$, then $k_1 e^{\mathbf{i}Z} k_1^{-1}=e^{\mathbf{i}Z} \ell_2k_1^{-1}$ which gives $e^{\mathbf{i} \mathrm{Ad}(k_1)(Z)}=e^{\mathbf{i}Z} (\ell_2k_1^{-1})$.
  By unicity of the polar decomposition, $\ell_2=k_1$ and $\mathrm{Ad}(k_1)(Z)=0$. Hence $k^{-1}\ell_1k=\ell_2$, so $k^{-1}L_1k \subset L_2$, which means that $[L_1] \preceq [L_2]$.
\end{proof}

\subsection{Real representations of compact Lie groups}
\label{sec:complexrep}

Let $K$ be a compact Lie group and $G$ be the complexification of $K$. If $\rho: K \to \GL(W)$ is a real continuous representation of a compact Lie group $K$, by proposition~\ref{def:universal-property}, there exists a unique analytic extension.
\begin{equation*}
  \rho^{\CC}: G \to \GL(V)
\end{equation*}
of $\rho$, where $V=W^{\mathbb{C}}$.

\begin{prop}  \cite[prop. III.8.6]{BD1985}, \cite[\S 5.4, cor. 1]{Ser1993}
  If we see $G$ as an affine algebraic group (via proposition \ref{prop:algebraic}), then $\rho^{\CC}$ is an algebraic map.
\end{prop}

\begin{prop} \cite[prop. 5.7]{Sch1980}, \cite[prop. 6.4]{BD1985}
  A representation of $G$ on a complex vector space $V$ is isomorphic to the complexification of representation of $K$ if and only if there exists a $G$-invariant nondegenerate complex bilinear form on $V$ \emph{(}we say that the representation is orthogonalizable\emph{)}.
\end{prop}

\begin{exmp}
  The representation of $\mathbb{C}^{\times}$ given in example \ref{baby} lets the form $B(z,w)=zw$ invariant. It is the complexification of the standard representation of $\mathrm{SO}(2; \mathbb{R})$ on $\mathbb{R}^2$, the eigenspaces for the weights $\lambda$ and $\lambda^{-1}$ being the lines $\mathbb{C}(1, \mathbf{-i})$ and $\mathbb{C}(1, \mathbf{i}$).
\end{exmp}

\begin{exmp}
  The representation of $\mathrm{SL}(2; \mathbb{C})$ on the vector space $V_d$ of binary forms admits an invariant bilinear form, which is the transvectant of maximal degree  $(F, G)_d$, also called polarity pairing. This form is nondegenerate, symmetric if $d$ is even, and skew-symmetric if $d$ is odd. Hence $V_d$ is the complexification of a real representation of $\mathrm{SU}(2; \mathbb{C})$ if $d$ is even. If $d$ is odd, this is not the case : $V_d$ is of quaternionic type (see \cite[prop. 6.4]{BD1985}).
\end{exmp}

Before studying complexified orbits, we recall a linear algebra lemma that will be used several times.

\begin{lem}\label{lem:endo-stable}
  Let $u$ be a diagonalizable endomorphism of a vector space $E$ with real eigenvalues. If $F\subset E$ is stable by $\mathrm{exp}(u)$, then $F$ is stable by $u$.
\end{lem}

\begin{proof}
  Let $\lambda_{1}, \ldots, \lambda_k$ be the distinct eigenvalues of $u$, and let $P$ be a polynomial such that for any $i$, $P(e^{\lambda_i})=\lambda_i$. Then $P(\mathrm{exp}(u))=u$, since this equality can be checked on a basis of eigenvectors of $u$.
\end{proof}

We fix a real representation $(K, W)$.
\begin{prop}\label{prop:orbites}
  For any $w$ in $W$, the following properties hold :
  \begin{enumerate}
    \item[(i)] $G_{w}=(K_{w})^{\CC}$,
    \item[(ii)] $G \cdot w$ is closed \cite[lem. 2.2]{EJ2009},
    \item[(iii)] $(G \cdot w) \cap W = K \cdot w$.
  \end{enumerate}
\end{prop}

\begin{proof}
  (i) by corollary~\ref{cor:totrel}, $K_{w}$ is Zariski dense in $(K_{w})^{\CC}$ so $(K_{w})^{\CC}$ fixes $w$ for the complexified representation. This yields the inclusion $(K_{w})^{\CC} \subset G_{w}$. For the converse implication, let us denote by $d\rho$ the differential of the $K$-action at the origin and let $\ell=ke^{\mathbf{i}Z}$ be an element of $G_{w}$. If $\iota$ is a Cartan involution, $\iota(\ell)$ belongs to $G_{w}$ too, so $e^{2\mathbf{i}Z}=\iota(\ell)^{-1} \ell$ belongs to $G_{w}$. It means that $\exp(2\mathbf{i}\,d\rho(Z))(w)=w$. The element $2\mathbf{i} \,d\rho(Z)$ is a complex hermitian endomorphism of $W$, so it is diagonalizable with real eigenvalues. Thanks to lemma \ref{lem:endo-stable}, $d\rho(Z)(w)=0$, so $Z$ belongs to $\mathrm{Lie}(G_{w})$. Hence, $e^{\mathbf{i}Z}$ fixes $w$ so $k$ fixes $w$ too. It follows that $k$ belongs to $K_{w}$, which proves that $\ell$ belongs to $(K_{w})^{\CC}$.

  (ii) Let us fix an inner product on $W$ such that $K$ acts by orthogonal transformations, and extend it to an hermitian product on $V$. The moment map $\mu \colon V \rightarrow \mathfrak{k}^*$ is given by
  \begin{equation*}
    \mu(v)(\xi)=\frac{1}{2\mathbf{i}} \langle \xi \cdot v | v \rangle
  \end{equation*}
  This definition makes sense because $\xi$ belongs to $\mathfrak{u}(V)$ so it is skew-hermitian (so the above expression is real). Kempf-Ness theorem states that an orbit $G^{\CC} \cdot v$ is closed if and only $G^{\CC} \cdot v \cap \mu^{-1}(0) \neq \emptyset$ (see for instance \cite{KN1979}, \cite[thm. 4]{Los2006}). However, since $K$ acts by orthogonal transformations, the moment map $\mu$ vanishes on $W$. The result follows.

  (iii) One inclusion is obvious. For the other inclusion, let $L=K_w$. The orbit $G \cdot w$ is isomorphic to $G/L^{\CC}$. Let $\iota$ denotes the Cartan involution on $G^{\CC}$. Assume that $g.w$ belongs to $W$ for some element $g$ in $G$. Then $\iota(g).w=\overline{g.w}=g.w$ so $\iota(g)^{-1}g$ belongs to $G_w=L^{\CC}$ by (i). If $g=ke^{iZ}$, then $e^{2\mathbf{i}Z}$ belongs to $L^{\CC}$. It implies that $Z$ is in $\mathfrak{l}$, so $e^{\mathbf{i}Z}$ is in $L^{\CC}$. Hence $g.w = k.w$.
\end{proof}

\begin{rem}
  The result~\cite[prop. 2.3]{BH1962} predicts in our case that the intersection of a complex orbit with real points is a \textit{finite union} of real orbits. Hence this result is weaker, but holds for more general groups.
\end{rem}

\begin{exmp}\label{ex:SL2R}
  Let us give an example where this intersection is strictly bigger than the real orbit. For this we take again Example~\ref{ex:sournois}, but with the action of $\mathrm{GL}(2; \CC)$ instead of $\mathrm{SL}(2; \CC)$ acting on the vector space of degree $3$ binary forms. Then the $\mathrm{GL}(2; \RR)$-orbit of $z(w-z)w$ consists of all binary forms of degree $3$ with real coefficients and 3 distinct roots. However, real binary forms that are in the $\mathrm{GL}(2; \CC)$ orbit of $z(w-z)w$ consist of real binary forms with 3 distinct roots (not necessarily real). This locus is a union of two $\mathrm{GL}(2; \RR)$ orbits: the orbit of $z(w-z)w$ and the orbit of $z(z^2+w^2)$. There is no contradiction with the result we proved, because $\mathrm{GL}(2; \RR)$ is not compact. As a matter of fact, $\mathrm{GL}(2; \RR)$ is the split real form of $\mathrm{GL}(2; \CC)$ and not the compact one, which is $\mathrm{U}(2)$.
\end{exmp}

\begin{prop}\label{prop:delicat} \cite[prop. 5.8 (2), (3)]{Sch1980}
  Given $(K, W, \rho)$, the following properties hold~:
  \begin{enumerate}
    \item[(i)] Orbit types of $\rho^{\mathbb{C}}$ are in bijection with complexifications of orbit types of $\rho$.
    \item[(ii)] For any real orbit type $[L]$ of $\rho$,
          $
            W^{\langle L \rangle} = V^{\langle L^{\mathbb{C}}\rangle} \cap W.
          $
          In particular, $W^{\langle L \rangle}$ is a nonempty Zariski open subset of $W^L$.
    \item[(iii)] There exists a $G$-stable and nonempty Zariski open subset $U$ of $V$ such that all orbits of points in $U$ are closed.
  \end{enumerate}
\end{prop}

\begin{proof}
  By Matsushima's criterion, any orbit type is reductive, so it is the complexification of a compact subgroup of $G$. This compact subgroup is included in a maximal compact subgroup of $G$, and since all these maximal compact subgroups are conjugate, we can assume that the orbit type is of the form $[L^{\mathbb{C}}]$ where $L$ is a subgroup of $K$. It remains to prove that $[L]$ is an orbit type of $\rho$. To prove it, we consider $V^{\langle L^{\mathbb{C}} \rangle}$, which is Zariski open in $V^{G}$. Since $V^{L^{\mathbb{C}}}$ is the complexification of $W^L$, $W^L$ is Zariski dense in $V^{L^{\mathbb{C}}}$ so $V^{\langle L^{\mathbb{C}}\rangle} \cap W^L \neq \emptyset$. If $w$ is in this intersection, $K_w=G_w \cap K = L^{\mathbb{C}} \cap K = L$. Conversely, if $w$ is a vector in $W$ such that $K_{w}=L$, thanks to proposition~\ref{prop:orbites} (ii), $G \cdot w $ is closed, and thanks \textit{loc. cit.} (i), $G_w=L^{\mathbb{C}}$. Hence $L^{\mathbb{C}}$ is a complex orbit type and $w$ belongs to $V^{\langle L^{\mathbb{C}}\rangle}$. This gives (i) as well as the first part of (ii). The second part follows from proposition \ref{prop:zdens}. For (iii), let $[H]$ be the minimal isotropy class of $\rho$. Then $[H^{\mathbb{C}}]$ is also the minimal isotropy class of $\rho^{\mathbb{C}}$, so points with closed orbits having isotropy $[H]$ is a Zariski open subset of $V$.
\end{proof}

\begin{rem} \label{rem:sink}
  Proposition \ref{prop:delicat} (i) is the same as \cite[prop. 5.8 (2)]{Sch1980}, but the proof given here is incomplete. The first part of the proof is exactly the same : if $[L^{\mathbb{C}}]$ is an orbit type, we have
  \[
    \emptyset \neq V^{\langle L^{\mathbb{C}}\rangle} \cap W^L \subset W^{\langle L \rangle}.
  \]
  The author then adds that $Z_{[L^{\CC}]}\cap S \subset S_{[L]}$, which has no reason a priori to be true without an additional argument. Indeed, thanks to proposition \ref{prop:direct}, $Z_{[L^{\CC}]}=p(V^{L^{\CC}})$. Hence an element in $Z_{[L^{\CC}]} \cap S \subset S_{[L]}$ can be written as $p(w)$ where $w$ is a vactor of $W$ such that $G \cdot w$ is closed and $G_{g.w}=L^{\CC}$, so $K_{w}^{\CC}=g^{-1}L^{\CC}g$ and it is needed to prove that $K_w$ is conjugate to $L$ in $K$. This is nevertheless true thanks to proposition \ref{prop:vaudaine}.
\end{rem}

Let $p \colon W \rightarrow W / K$ be the natural projection and $S$ its image inside $V // G$, $S$ is a semi-algebraic subset of $Z$.

\begin{prop} \label{prop:vengeance} \cite[prop. 5.8 (3)]{Sch1980}, see remark \ref{rem:sink}.
  Let $(W, K, \rho)$ be a real representation of a compact group $K$. If $[L]$ is an orbit type of $\rho$ and $S_{[L]}=p(\Sigma_{[L]})$, then\emph{:}
  \begin{enumerate}
    \item[(i)] $S_{[L]}=Z_{[L^{\CC}]} \cap S$,
    \item[(ii)] $\Sigma_{[L]}=\Omega_{[L^{\CC}]} \cap W$.
  \end{enumerate}
\end{prop}

\begin{proof}
  The inclusion $S_{[L]}\subset Z_{[L^{\CC}]} \cap S$ follows from proposition \ref{prop:orbites}. For the reverse inclusion, let $G=K^{\CC}$ and $H=L^{\CC}$. A point in $Z_{[L^{\mathbb{C}}]} \cap S$ can be written as $p(w)$ where $w$ is in $W$ and $G_{g.w}=H$ for some $g$ in $G$. Thus, $G_w=g^{-1}Hg$ so $[G_w]=[H]$. Thanks to proposition \ref{prop:vaudaine} and proposition \ref{prop:orbites}, $[K_w]=[L]$ so $w$ belongs to $S_{[L]}$. The second point follows from the first using proposition \ref{prop:direct} (iv).
\end{proof}

\begin{lem}\label{lem:complexification} \cite[prop. 5.8 (1)]{Sch1980}
  The three complex algebras $(\RR[W]^{K})^{\CC}$, $\CC[V]^{K}$ and $\CC[V]^{G}$ are naturally isomorphic.
\end{lem}

\begin{proof}
  We have $\RR[W]^{\CC}\simeq \CC[V]$. Then the result follows from corollary~\ref{cor:cx}.
\end{proof}
%
%
%
%
%

\subsection{Normalizers of subgroups of compact Lie groups}

The aim of this section is to prove that the complexification of the normalizer of a compact subgroup of a compact Lie group is naturally isomorphic to the normalizer of the corresponding complexified subgroup (proposition~\ref{prop:NH-complexification}). This result is in accordance with \cite[lem. 1.1]{LR1979}.

\begin{lem}\label{lem:Lie-NH}
  Let $L$ be a closed subgroup of a compact Lie group $K$ and let $\Gamma$ be a set of representatives in $L$ of the finite group $L/L^{e}$, where $L^{e}$ is the identity component of $L$. Then,
  \begin{equation*}
    \mathrm{Lie}(N_K(L)) = N_{\mathfrak{k}}(\mathfrak{l})\cap\set{X\in \mathfrak{k},\ \forall \gamma\in \Gamma,\ \Ad(\gamma)(X)-X\in \mathfrak{l}}.
  \end{equation*}
\end{lem}

\begin{proof}
  We have $\mathrm{Lie}(N_{K^e}(L^e))=N_{\mathfrak{k}}(\mathfrak{l})$, and since the group $(N_K(L))^e$ is included in $N_{K^e}(L^e)$, $\mathrm{Lie}(N_K(L))$ is included in  $N(\mathfrak{l})$. For $\gamma$ in $\Gamma$ and $X$ in $\mathfrak{k}$, $e^{tX} \gamma e^{-tX}$ belongs to the connected component $L^e \gamma$ of $L$ containing $\gamma$. Consequently, $e^{tX}\gamma e^{-tX} \gamma^{-1}\in L$. Taking the derivative at $t=0$, we get that $X-\Ad(\gamma)(X)$ belongs to $\mathfrak{l}$.

  Conversely, let $X\in N_{\mathfrak{k}}(\mathfrak{l})\cap\set{X\in \mathfrak{k},\ \forall \gamma\in \Gamma,\ \Ad(\gamma)(X)-X\in \mathfrak{l}}$. Since $X$ belongs to $N_{\mathfrak{k}}(\mathfrak{l})$,  $e^{tX}L^e e^{-tX} \subset L^e$. Let $\gamma$ be in $\Gamma$. Then
  \begin{equation*}
    e^{tX}\gamma e^{-tX}\gamma^{-1}=e^{tX}e^{-\Ad(\gamma)(tX)}
    =e^{tX}e^{t(X-\Ad(\gamma)(X))-tX}.
  \end{equation*}
  Let $Y=X-\Ad(\gamma)(X) \in \mathfrak{l}$. By applying Baker-Campbell-Hausdorff formula on $e^{tX} e^{tY-tX}$, we obtain that for $t$ sufficiently small $e^{tX}	\gamma e^{-tX}\gamma^{-1}$ is equal to $e^{Z}$ where $Z$ is the sum of $tX+(tY-tX)$ and a convergent sum of iterated brackets of $X$ and $Y$. Hence $Z$ belongs to $\mathfrak{l}$, which implies that $e^{tX}\gamma e^{-tX}\gamma^{-1}$ belongs to $L^e$ for $t$ sufficiently small, then for all $t$. Hence for any $\eta$ in $L^e$,
  \begin{equation*}
    e^{tX}(\gamma \eta) e^{-tX}=\underbrace{(e^{tX}\gamma e^{-tX} \gamma^{-1})}_{\in L^e} \gamma \underbrace{(e^{tX}\eta e^{-tX})}_{\in L^e}\in L^e \gamma L^e \subset L.
  \end{equation*}
  Consequently, $e^{tX}$ normalizes $L$, which implies that $X$ belongs to $\mathrm{Lie}(N_K(L))$.
\end{proof}

\begin{lem}\label{lem:A-B}
  Let $\mathfrak{l}$ be a Lie subalgebra of $\mathfrak{u}(n)$, $Z\in \mathfrak{l}$, $B\in N_{\mathfrak{u}(n)}(\mathfrak{l})$. Then $e^{\mathbf{i}B}e^{\mathbf{i}Z}e^{\mathbf{i}B}$ is hermitian positive definite, and if we write its polar decomposition as $e^{2\mathbf{i}A}$ where $A$ is in $\mathfrak{u}(n)$, then $A-B\in \mathfrak{l}$.
\end{lem}

\begin{proof}
  Consider the following real analytic function
  \begin{align*}
    f \colon \mathfrak{u}(n) \times N_{\mathfrak{u}(n)}(\mathfrak{l}) & \to \mathfrak{u}(n) \\ (Z,B) &\mapsto \frac{1}{\mathbf{i}}\log (e^{\mathbf{i}B}e^{\mathbf{i}Z}e^{\mathbf{i}B})-2B.
  \end{align*}
  We apply Baker-Campbell-Hausdorff formula two times on $e^{\mathbf{i}B}e^{\mathbf{i}Z}e^{\mathbf{i}B}$ for $(Z, B)$ in a neighborhood of $0$. It yields $e^{\mathbf{i}Z}e^{\mathbf{i}B}=e^{iW}$ where $W \in \mathbf{i}(Z+B) + \mathfrak{l}^{\CC} = \mathbf{i}B + \mathfrak{l}^{\CC}$, and then $e^{\mathbf{i}B}e^{\mathbf{i}W}=e^{\mathbf{i}Y}$ where $Y\in \mathbf{i}(W+B)+\mathfrak{l}^{\CC}=2\mathbf{i}B + \mathfrak{l}^{\mathbb{C}}$. Hence,
  \[
    \frac{1}{\mathbf{i}}\log (e^{\mathbf{i}B}e^{\mathbf{i}Z}e^{\mathbf{i}B}) \in (2B + \mathfrak{l}^{\CC}) \cap \mathfrak{u}(n)=2B + \mathfrak{l}
  \]
  so $f(Z,B)$ takes values in $\mathfrak{l}$ for $B$ and $Z$ sufficiently close to zero. Since $f$ is real analytic, this is valid everywhere.
\end{proof}

\begin{prop}[see {\cite[lem. 1.1]{LR1979}}] \label{prop:NH-complexification}
  If $L$ is a closed subgroup of a compact Lie group $K$, let $G=K^{\CC}$ and let $H=L^{\CC}$. Then $
    N_{G}(H)=(N_{K}(L))^{\CC}.$
\end{prop}

\begin{proof}
  If $k$ is in $N_{K}(L)$, then, $kLk^{-1} =  L \subset H$. Consider the holomorphic application
  \begin{align*}
    C(k) \colon H & \to G/H           \\
    h             & \mapsto khk^{-1}.
  \end{align*}
  Then, $C(k)$ vanishes on $L$, and thanks to corollary~\ref{cor:totrel}, $C(k)$ vanishes on $H$. Hence $kHk^{-1} \subset H$ so $k$ normalizes $H$ in $G$. Hence $N_{K}(L)\subset N_{G}(H)$ and by theorem~\ref{prop:Chevalley-complexification} we get the natural inclusion
  \begin{equation*}
    (N_{K}(L))^{\CC}\subset N_{G}(H).
  \end{equation*}
  To prove the converse inclusion, let $g$ be in $N_{G}(H)$. By theorem~\ref{prop:Chevalley-complexification}, we can write $g=ke^{\mathbf{i}X}$ with $k$ in $K$ and $X$ in $\mathfrak{k}$. We need to prove that $k$ (resp. $X$) belongs to $N_K(L)$ (resp. to $\mathrm{Lie}(N_K(L))$. To prove that $X\in \mathrm{Lie}(N_K(L))$, we use the description of $\mathrm{Lie}(N_K(L))$ given in lemma~\ref{lem:Lie-NH}.

  On one hand, we have $\iota(g)=ke^{-\mathbf{i}X}$ so $e^{2\mathbf{i}X}=\iota(g)^{-1}g$. Besides, the normalizer of $H$ is stable under the Cartan involution $\iota$, and hence $e^{2\mathbf{i}X}$ normalizes $H$. This implies that $\Ad(e^{2\mathbf{i}X})(\mathfrak{h}^{\CC})= \mathfrak{h}^{\CC}$, where $\mathfrak{h}$ is the Lie algebra of $H$. Now $\mathrm{Ad}(e^{2\mathbf{i}X})=\mathrm{Exp}(2\mathbf{i} \mathrm{ad}(X))$. Since $X$ is in $\mathfrak{u}(n)$, $2\mathbf{i} \mathrm{ad}(X)$ is hermitian, hence diagonalizable with real eigenvalues. Thanks to lemma~\ref{lem:endo-stable}, $\mathrm{ad}(X)(\mathfrak{h}^{\CC})\subset \mathfrak{h}^{\CC}$. Since $X$ is in $\mathfrak{u}(n)$, $X$ belongs to $ N_{\mathfrak{g}}(\mathfrak{h})$.

  On the other hand, let $\gamma$ in $\Gamma$. Then $e^{2\mathbf{i}X}\gamma e^{-2\mathbf{i}X}$ belongs to $H$. Let $e^{\mathbf{i}Z}$ be the hermitian part of its polar decomposition, where $Z \in \mathfrak{h}$. Then $(\iota(e^{2\mathbf{i}X}\gamma e^{-2\mathbf{i}X}))^{-1}(e^{2\mathbf{i}X}\gamma e^{-2\mathbf{i}X})=e^{2\mathbf{i}Z}$ but by explicit calculation,
  \begin{align*}
    (\iota(e^{2\mathbf{i}X}\gamma e^{-2\mathbf{i}X}))^{-1}(e^{2\mathbf{i}X}\gamma e^{-2\mathbf{i}X}) & = (e^{-2\mathbf{i}X}\gamma e^{2\mathbf{i}X})^{-1}(e^{2\mathbf{i}X}\gamma e^{-2\mathbf{i}X}) \\
                                                                                                     & = e^{-2\mathbf{i}X} \left(\gamma^{-1} e^{4\mathbf{i}X} \gamma\right) e^{-2\mathbf{i}X}      \\
                                                                                                     & = e^{-2\mathbf{i}X}e^{4\mathbf{i}\Ad(\gamma^{-1})(X)}e^{-2\mathbf{i}X}.
  \end{align*}
  Therefore, combining the two equalities together, we get
  \begin{equation*}
    e^{2\mathbf{i}Z}=e^{-2\mathbf{i}X}e^{4\mathbf{i}\Ad(\gamma^{-1})(X)}e^{-2\mathbf{i}X}
  \end{equation*}
  that is
  \begin{equation*}
    e^{4\mathbf{i}\Ad(\gamma^{-1})(X)}=e^{2\mathbf{i}X}e^{2\mathbf{i}Z}e^{2\mathbf{i}X}.
  \end{equation*}
  Applying lemma~\ref{lem:A-B} with $A=2\Ad(\gamma^{-1})(X)\in \mathfrak{u}(n)$ and $B=2X\in N_{\mathfrak{k}}(\mathfrak{h})$, we have finally $\Ad(\gamma^{-1})(X)-X\in \mathfrak{l}$, so $X-\Ad(\gamma)(X) \in \mathfrak{l}$. This proves that $X$ belongs to $\mathrm{Lie}(N_K(L))$.
  Now, both $e^{\mathbf{i}X}$ and $ke^{\mathbf{i}X}$ normalize $H$ and so does $k$. Hence $kHk^{-1}=H$ so
  \begin{equation*}
    L=H \cap \mathrm{U}(n)=kHk^{-1} \cap \mathrm{U}(n)=k (H\cap \mathrm{U}(n))k^{-1}=kLk^{-1}
  \end{equation*}
  and $k$ normalizes $L$.
\end{proof}

\section{Proof of the main theorems}
\label{sec:proof-main-theorems}

\subsection{The algebraicity theorem (theorem \ref{thm:first-main-theorem})}

We prove a more precise result than theorem \ref{thm:first-main-theorem}, namely that $\overline{\Sigma}_{[L]}$ equals $\Lambda_{[L^{\CC}]} \cap W$. This means that the closed real stratum consists of the real points of the closed complex stratum.

In one direction, any point $w$ in $\overline{\Sigma}_{[L]}$ belongs to some stratum $\Sigma_{[M]}$ where $[L] \preceq [M]$. Since $w$ has a closed orbit with stabilizer $M^{\CC}$ and since $[L^{\CC}] \preceq [M^{\CC}]$, $w$ belongs to $\Lambda_{[L^{\CC}]}$.

Conversely, assume that an element $w$ of $W$ belongs to $\Lambda_{[L^{\CC}]}$. Then the orbit of $w$ is also closed, and $[G_w]=[M^{\CC}]$ where $[M]$ is a real orbit type, and $[L^{\CC}] \preceq [M^{\CC}]$. By proposition \ref{prop:vaudaine}, $[L] \preceq [M]$ so $w$ belongs to $\Sigma_{[M]}$, hence to $\overline{\Sigma}_{[L]}$.

\subsection{The normalization theorem (theorem \ref{thm:second-main-theorem})}
We consider the restriction morphism
\[
  \RR[W]^{K}  \to \RR[W^{L}]^{N_K(L)}.
\]
Thanks to lemma \ref{lem:complexification} and proposition \ref{prop:NH-complexification}, its complexification is the map
\[
  \RR[V]^{G}  \to \RR[V^{H}]^{N_G(H)}.
\]
where $G=K^{\CC}$ and $H=L^{\CC}$. Hence the complexification of the map we are interested in is the map
\[
  \mathrm{Im}\left(\CC[V]^{G}  \to \CC[V^{H}]^{N_G(H)}\right) \rightarrow \CC[V^{H}]^{N_G(H)}
\]
We look at the composition
\[
  \CC[V]^G \rightarrow \CC \left[\overline{G.V^H} \right]^G \rightarrow \CC[V^H]^{N_G(H)}.
\]
Since $G$ is reductive, the first arrow is surjective. The second arrow is injective. Hence we have a chain of morphisms
\[
  \CC \left[\overline{G.V^H} \right]^G \xrightarrow{\sim} \mathrm{Im}\left(\CC[V]^{G}  \to \CC[V^{H}]^{N_G(H)}\right) \rightarrow \CC[V^{H}]^{N_G(H)}
\]
and the corresponding morphism of affine varieties is
\[
V^H // N_G(H) \rightarrow \overline{G.V^H} // G.
\]
Thanks to proposition \ref{prop:normal}, this map is a normalization. Hence the initial map is also a normalization.



\begin{thebibliography}{10}

\bibitem{AS1983}
M.~Abud and G.~Sartori.
\newblock The geometry of spontaneous symmetry breaking.
\newblock {\em Annals of Physics}, 150(2):307--372, Oct. 1983.

\bibitem{AKP2013}
N.~Auffray, B.~Kolev, and M.~Petitot.
\newblock On anisotropic polynomial relations for the elasticity tensor.
\newblock {\em Journal of Elasticity}, 115(1):77--103, June 2013.

\bibitem{Bor1960}
A.~Borel.
\newblock {\em Seminar on transformation groups}.
\newblock Annals of Mathematics Studies, No. 46. Princeton University Press,
  Princeton, N.J., 1960.
\newblock With contributions by G. Bredon, E. E. Floyd, D. Montgomery, R.
  Palais.

\bibitem{BH1962}
A.~Borel and Harish-Chandra.
\newblock Arithmetic subgroups of algebraic groups.
\newblock {\em The Annals of Mathematics}, 75(3):485, May 1962.

\bibitem{BB1958}
N.~Bourbaki and N.~Bourbaki.
\newblock {\em Éléments de mathématique}, volume~27.
\newblock Hermann Paris, 1958.

\bibitem{Bre1972}
G.~E. Bredon.
\newblock {\em Introduction to compact transformation groups}.
\newblock Pure and Applied Mathematics, Vol. 46. Academic Press, New
  York-London, 1972.

\bibitem{BD1985}
T.~Bröcker and T.~Dieck.
\newblock {\em Representations of Compact Lie Groups}.
\newblock Springer Berlin Heidelberg, 1985.

\bibitem{DK2000}
J.~J. Duistermaat and J.~A.~C. Kolk.
\newblock {\em Lie groups}.
\newblock Universitext. Springer-Verlag, Berlin, 2000.

\bibitem{EJ2009}
P.~Eberlein and M.~Jablonski.
\newblock Closed orbits of semisimple group actions and the real
  hilbert-mumford function.
\newblock {\em Contemporary Mathematics}, 491:283--321, 2009.

\bibitem{FV1996}
S.~Forte and M.~Vianello.
\newblock Symmetry classes for elasticity tensors.
\newblock {\em Journal of Elasticity}, 43(2):81--108, 1996.

\bibitem{Gro1958}
A.~Grothendieck.
\newblock Torsion homologique et sections rationnelles.
\newblock {\em Séminaire Claude Chevalley}, 3:1--29, 1958.

\bibitem{Gro1967}
A.~Grothendieck.
\newblock Éléments de géométrie algébrique. {IV}. Étude
  locale des schémas et des morphismes de schémas {IV}.
\newblock {\em Inst. Hautes Études Sci. Publ. Math.}, (32):361, 1967.

\bibitem{Hel2001}
S.~Helgason.
\newblock {\em Differential geometry, {L}ie groups, and symmetric spaces},
  volume~34 of {\em Graduate Studies in Mathematics}.
\newblock American Mathematical Society, Providence, RI, 2001.
\newblock Corrected reprint of the 1978 original.

\bibitem{Hil1993}
D.~Hilbert.
\newblock {\em Theory of algebraic invariants}.
\newblock Cambridge University Press, Cambridge, 1993.
\newblock Translated from the German and with a preface by Reinhard C.
  Laubenbacher, Edited and with an introduction by Bernd Sturmfels.

\bibitem{Hubert_Jalard_2025}
E.~Hubert and M.~Jalard.
\newblock Orbit separation and stratification by isotropy classes of
  piezoelectricity tensors.
\newblock {\em Journal of Pure and Applied Algebra}, 229(9):108034, 2025.

\bibitem{KN1979}
G.~Kempf and L.~Ness.
\newblock The length of vectors in representation spaces.
\newblock In {\em Lecture Notes in Mathematics}, volume 732 of {\em Lecture
  Notes in Math.}, pages 233--243. Springer Berlin Heidelberg, 1979.

\bibitem{LPot1997}
J.~Le~Potier.
\newblock {\em Lectures on vector bundles}, volume~54 of {\em Cambridge Studies
  in Advanced Mathematics}.
\newblock Cambridge University Press, Cambridge, 1997.
\newblock Translated by A. Maciocia.

\bibitem{LePotier_1997}
J.~Le~Potier.
\newblock {\em Lectures on vector bundles}.
\newblock Cambridge studies in advanced mathematics 54. Cambridge University
  Press, Cambridge, 1997.

\bibitem{Lee2002}
D.~H. Lee.
\newblock {\em The structure of complex {L}ie groups}, volume 429 of {\em
  Chapman \& Hall/CRC Research Notes in Mathematics}.
\newblock Chapman \& Hall/CRC, Boca Raton, FL, 2002.

\bibitem{Los2006}
I.~V. Losev.
\newblock The {Kempf-Ness} theorem and {Invariant Theory}.
\newblock 2006.

\bibitem{Lun1972}
D.~Luna.
\newblock Sur les orbites fermées des groupes algébriques
  réductifs.
\newblock {\em Invent. Math.}, 16:1--5, 1972.

\bibitem{Lun1973}
D.~Luna.
\newblock Slices étales.
\newblock In {\em Sur les groupes algébriques}, Bull. Soc. Math. France,
  Paris, Mémoire 33, pages 81--105. Soc. Math. France, Paris, 1973.

\bibitem{Lun1975}
D.~Luna.
\newblock Adhérences d{\textquotesingle}orbite et invariants.
\newblock {\em Inventiones Mathematicae}, 29(3):231--238, Oct. 1975.

\bibitem{LR1979}
D.~Luna and R.~W. Richardson.
\newblock A generalization of the chevalley restriction theorem.
\newblock {\em Duke Mathematical Journal}, 46(3):487--496, Sept. 1979.

\bibitem{OKDD2021}
M.~Olive, B.~Kolev, R.~Desmorat, and B.~Desmorat.
\newblock Characterization of the symmetry class of an elasticity tensor using
  polynomial covariants.
\newblock {\em Mathematics and Mechanics of Solids}, 27(1):144--190, may 2021.

\bibitem{Popov1992}
V.~L. Popov.
\newblock On the ``{Lemma} of {Seshadri}''.
\newblock In {\em Lie groups, their discrete subgroups, and invariant theory},
  pages 167--172. Providence, RI: American Mathematical Society, 1992.

\bibitem{PS1985}
C.~Procesi and G.~Schwarz.
\newblock Inequalities defining orbit spaces.
\newblock {\em Inventiones mathematicae}, 81(3):539--554, 1985.

\bibitem{Richardson1972}
R.~J. Richardson.
\newblock Principal orbit types for algebraic transformation spaces in
  characteristic zero.
\newblock {\em Inventiones mathematicae}, 16:6--14, 1972.

\bibitem{Sch1980}
G.~W. Schwarz.
\newblock Lifting smooth homotopies of orbit spaces.
\newblock {\em Publications math{é}matiques de
  l{\textquotesingle}{IH}{É}S}, 51(1):37--132, Dec. 1980.

\bibitem{Ser1993}
J.-P. Serre.
\newblock G\`ebres.
\newblock {\em Enseign. Math. (2)}, 39(1-2):33--85, 1993.

\bibitem{Ser2006}
J.-P. Serre.
\newblock {\em Lie algebras and {L}ie groups}, volume 1500 of {\em Lecture
  Notes in Mathematics}.
\newblock Springer-Verlag, Berlin, 2006.
\newblock 1964 lectures given at Harvard University, Corrected fifth printing
  of the second (1992) edition.

\bibitem{tDie1987}
T.~tom Dieck.
\newblock {\em Transformation Groups}, volume~8 of {\em De Gruyter Studies in
  Mathematics}.
\newblock {DE} {GRUYTER}, Dec. 1987.

\end{thebibliography}
\end{document}